\documentclass[12pt]{article}
\title{The Lightning Model}
\author{James T. Campbell\footnote{Corresponding Author: jcampbll@memphis.edu. Partially supported by a University of Memphis Hardin Honors College Summer Research Grant} \and Alexandra Deane\footnote{Partially supported by an NSERC-USRA grant} \and Anthony Quas\footnote{Partially supported by an NSERC grant}}

\newcommand{\Addresses}{{
		\bigskip
		\footnotesize
		
		J.T.~Campbell$^*$ (Corresponding author), \textsc{Department of Mathematical Sciences, University of Memphis, Memphis TN 38152
			}\par\nopagebreak
		\textit{E-mail address}, J.T.~Campbell: \texttt{jcampbll\tocat memphis.edu}
		
		\medskip
		
		A.~Deane, \textsc{Department of Mathematics and Statistics, University of Victoria, Victoria, BC V8W 2Y2	Canada 
			}\par\nopagebreak
		\textit{E-mail address}, A.~Deane: \texttt{alexandradeane\tocat uvic.ca}
		
		\medskip
		
		A.~Quas, \textsc{Department of Mathematics and Statistics, University of Victoria, Victoria, BC V8W 2Y2 	
			Canada
			}\par\nopagebreak
		\textit{E-mail address}, A.~Quas: \texttt{aquas\tocat uvic.ca}
		
}}

\usepackage{amsmath, comment}
\usepackage{amsfonts, relsize}
\usepackage{amssymb}
\usepackage{amsthm}
\usepackage{marvosym}
\usepackage{mathtools}
\usepackage{mathrsfs}
\usepackage{dsfont}
\usepackage{graphicx}
\usepackage{subcaption, tikz}

\newcommand{\ep}{\ensuremath{\epsilon}}

\newcommand{\fep}{f_\ep}
\newcommand{\fepp}{f_{\ep'}}
\newcommand{\epp}{{\ep'}}
\newcommand{\V}{\ensuremath{\mathbb V} }
\newcommand{\E}{\mathbb{E}}
\newcommand{\Z}{\mathbb Z}

\newcommand{\tocat}{\includegraphics[scale=0.0234]{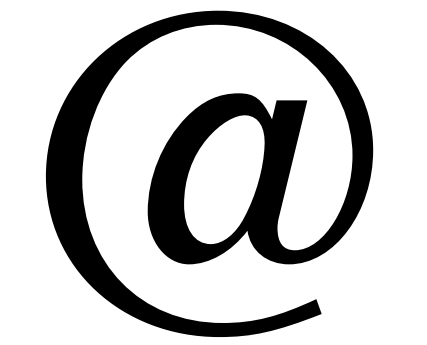}}

\DeclarePairedDelimiter{\ceil}{\lceil}{\rceil}
\DeclarePairedDelimiter{\floor}{\lfloor}{\rfloor}
\newcommand{\N}{\mathbb{N}}
\newcommand{\Pep}{\mathbb{P}_\epsilon}

\newcommand{\gep}{\overset{\epsilon}{>}}
\newcommand{\tep}{\overset{\epsilon}{\to}}
\newcommand{\lrep}{\overset{\epsilon}{\leftrightarrow}}

\newcommand{\Fep}{F^{\epsilon}}
\newcommand{\Lep}{\mathscr{L}^{\epsilon}}
\newcommand{\Fspace}{\mathcal{F}^{\epsilon}}
\newcommand{\restr}[1]{\raisebox{-.5ex}{$\big |$}_{#1}}
\newcommand{\1}[1]{\mathds{1}_{#1}}
\newcommand{\Nom}{N_\omega}

\newcommand{\Ph}{\Psi_{n}}
\newcommand{\fle}{\ensuremath{\floor{\frac{1}{\epsilon}}} }
\theoremstyle{plain}
\newtheorem{theorem}{Theorem}[section]
\newtheorem*{theorem*}{Theorem}
\newtheorem{lemma}[theorem]{Lemma}
\newtheorem{proposition}[theorem]{Proposition}
\newtheorem*{corollary}{Corollary}

\theoremstyle{definition}
\newtheorem{definition}[theorem]{Definition}

\newtheorem{example}[theorem]{Example}

\theoremstyle{remark}
\newtheorem*{remark}{Remark}
\newtheorem*{claim}{Claim}

\begin{document}
\maketitle
\vspace{-.25in}
\begin{abstract}
	We introduce a non-standard model for percolation on the integer lattice $\Z^2$. Randomly assign to each vertex $a \in \Z^2$ a potential, denoted $\phi_a$,  chosen independently and uniformly from the interval $[0, 1]$. For fixed $\epsilon \in [0,1]$, draw a directed edge from vertex $a$ to a nearest-neighbor vertex $b$ if $\phi_b < \phi_a + \ep$, yielding a directed subgraph of the infinite directed graph $\overrightarrow{G}$ whose vertex set is $\Z^2$, with nearest-neighbor edge set. We define notions of weak and strong percolation for our model, and observe that when $\ep = 0$ the model fails to percolate weakly, while for $\ep = 1$ it percolates strongly. We show that there is a positive $\epsilon_0$ so that for $0 \le \ep \le \ep_0$, the model fails to percolate weakly, and that when $\ep > p_\text{site}$, the critical probability for standard site percolation in $\Z^2$, the model percolates strongly. We study the number of infinite strongly connected clusters occurring in a typical configuration. We show that for these models of percolation on directed graphs, there are some subtle issues that do not arise for undirected percolation. Although our model does not have the finite energy property, we are able to show that, as in the standard model, the number of infinite strongly connected clusters is almost surely 0, 1 or $\infty$.
	
\end{abstract}

\textit{2010 Mathematics Subject Classification}:  60K35, 82B43\medskip

\textit{Keywords:}  percolation, integer lattice,  phase transition, infinite clusters 

\section{Introduction} \label{sec:intro}

In this paper we introduce and establish some preliminary results about the following family of non-standard models for percolation on the directed integer lattice $\Z^2$.  Randomly assign a {\em potential} to each vertex in $\Z^2$, where the values are chosen independently and uniformly from the interval $[0,1]$. Such an assignment is called a {\em vertex configuration}; if $\phi$ is such a configuration and $a \in \Z^2$, we designate the value of $\phi$ at $a$ by $\phi_a$. Fix $\epsilon \in [0,1]$, and for nearest neighbor vertices $a$ and $b$, draw a directed edge from $a$ to $b$ if $\phi_b < \phi_a + \ep$, giving a subgraph of the nearest neighbor graph on $\Z^2$.
Thus, each vertex configuration gives rise to an {\em edge configuration}, and there is a natural probability measure (the push-forward of Lebesgue product measure on $[0,1]^{\Z^2}$) on the edge configuration space. The {\em Lightning Model} refers to this family (for $0 \le \ep \le 1$), or perhaps a fixed member of this family,  of configuration spaces. 


If there is an infinite path originating at the origin 0, we say the configuration {\em percolates weakly}. Define the {\em strong cluster} of 0 to be the strongly connected component containing 0, namely, the set of vertices $a$ for which there is a directed path from 0 to $a$ and also a directed path from $a$ to 0. When this cluster is infinite, we say that the configuration {\em percolates strongly}. It is immediate that strong percolation implies weak percolation.  

We are interested mainly in the question: for which values of $\ep$ does the Lightning Model percolate with positive probability?  

 In Section \ref{sec:lowerweak} we find an $\epsilon_0>0$ so that for $0 \le \ep < \epsilon_0$, weak percolation fails to occur. While soft arguments establish the existence of such an $\epsilon_0>0$, our arguments
(based on computing the spectral radius of certain operators) give an explicit non-trivial lower bound.
We also show that when $\ep > p_{\text{site}}$, the critical probability for standard site percolation on $\Z^2$, strong percolation occurs in the Lightning Model. 
The estimates in this section are not sharp, and there is a substantial gap between $\ep_0$ and $p_\text{site}$.  This leaves the question of determining  precise {\em critical values} $0 < \epsilon_w \le \epsilon_s < 1$ so that for $0 < \epsilon_w$, weak percolation does not occur, and for $\ep > \ep_w$, weak percolation occurs with positive probability. Similarly for $\ep_s$ and strong percolation. We conjecture that $\ep_w=\ep_s$.

In the standard (non-directed) site and bond percolation models, it is straightforward to show using ergodicity and the finite energy condition that for each $p$, there is $N_p\in\{0,1,\infty\}$ such that the number of infinite clusters for a.e.\ configuration is $N_p$ (first established in \cite{NS}). Furthermore a well-known argument of Burton and Keane (\cite{BK}) shows that $N_p$ only takes the values 0 and 1 in these models. These questions become far more subtle when studying strongly connected clusters in the Lightning Model. Firstly the model lacks the finite energy condition, but more seriously strong connectedness in directed graphs is a much more sensitive property than connectedness in undirected graphs. For example, changing a single edge can break a single strongly connected cluster into infinitely many finite clusters. In Section \ref{sec:numcom}, we establish that in the Lightning Model, the number of strongly connected clusters is almost surely 0, 1 or $\infty$. Our proof is strictly two-dimensional, and we do not know if the corresponding statement holds in the higher-dimensional version of the model. In Section \ref{sec:open}, we present some open questions, and give further examples highlighting the difficulties with Burton-Keane type arguments in directed settings.  As Grimmett (\cite{GG2}) indicates, ``The Burton-Keane method is canonical of its type, and is the first port of call in any situation
where such a uniqueness result is needed". It seems that for this model, a different sort argument must be used, which presents a truly interesting situation. 

We use the term {\em Lightning Model} because the base idea (since modified to its present form) originated as a simple model for lightning in which preferred transitions are possible based on relative values of a potential. We discovered later that in fact a model similar to the one in this paper, in three dimensions, had been proposed by climatologists (\cite{RPKTR}). Although their paper focuses on basic geometric properties of simulations in finite regions, the double connection confirmed our name choice.

\section{Preliminaries} \label{sec:first}
\subsection{Basic Definitions, Paths, Clusters} \label{sec:defn}

Our base graph is the infinite directed nearest-neighbor graph $\overrightarrow{G}$ whose vertex set is 
$\Z^2$, with edge set $E = \{(a, a \pm e_i): a \in \Z^2, \, i = 1, 2\}$, where $e_1$ 
and $e_2$ are the unit coordinate vectors.  Adjacent vertices $a, b \in \Z^2$ are  called {\em neighbors}.

\begin{definition}
 The set \V of \emph{vertex configurations}, is defined as 
\[ \mathbb{V} = [0,1]^{\mathbb{Z}^2},  \]
where $[0, 1]$ is the unit interval equipped with the usual topology. We equip $\mathbb V$ with the product topology and the Borel $\sigma$-algebra $\mathcal B$, and put Lebesgue product measure $\lambda$ on  $(\mathbb V, \mathcal B)$.

For a fixed vertex configuration $\phi \in \mathbb V$, we will write $\phi_a$ (or when clarity requires, $\phi(a)$) to denote the value of $\phi$ at vertex $a$, and call it the \emph{potential} at $a$.
\end{definition}

\begin{definition}
The set of \emph{edge configurations} is defined as
\[ \E = {\{\,  \raisebox{2pt}{\tikz\draw[black,fill=black] (0,0) circle (.3ex);}\; , \, \to \, \}}^E, \]
again equipped with the product topology and Borel $\sigma$-algebra $\mathscr B$.
\end{definition}

\begin{definition}
For a parameter $\epsilon \in [0,1]$, define $\fep: \V \to \E$ by  
\[ \fep(\phi) = z, \]
where for every edge $e = (a, b)$ in $E$, 
\[ z_e := z(a,b) =
\begin{cases}
\rightarrow 		& \text{ if } \phi_b < \phi_a +\ep; \\
\raisebox{2pt}{\tikz\draw[black,fill=black] (0,0) circle (.3ex);}	&\text{ otherwise.}
\end{cases} \]

We think of $\to$ as representing a directed edge and $ \, \raisebox{2pt}{\tikz\draw[black,fill=black] (0,0) circle (.3ex);}\,$ representing the absence of such an edge. Hence for a fixed $\phi$ and $\ep$, $\fep(\phi)$ represents an infinite directed subgraph of $\overrightarrow{G}$ with vertex set $\Z^2$, with a directed edge from the vertex $a$ to the vertex $b$ iff the potential at $b$ is lower than that at $a$, up to a tolerance of $\epsilon$.  We often use standard percolation terminology and call the directed edges in this subgraph {\em open}.
\end{definition}

\begin{definition} \label{def:pep}
For  $\ep \in [0,1]$ define a probability measure $\Pep$ on $(\E, \mathscr B)$ by 
\[ \Pep(A) := \lambda(\fep^{-1}(A)), \quad A \in \mathcal B, \]
where $\lambda$ is the Lebesgue product measure on $(\mathbb V, \mathcal B)$. 
\end{definition}

\begin{remark} An efficient way to describe our setup is to view the vertex potentials as a family $\{U_a\}_{a \in \Z^2}$ of i.i.d. standard uniform random variables. For fixed  $\ep > 0$,   declare the directed nearest neighbor edge $(a,b)$ to be $\ep$-open if $U_b < U_a + \ep$. 
\end{remark}

\begin{definition} By the {\em Lightning Model} we mean the aggregate of probability spaces $(\E, \mathscr B, \Pep)$, for $0 \le \ep \le 1$. We may also refer to the space with a fixed $\epsilon$ as the Lightning Model. 
	
\end{definition}

We want to study  the (typical) component structure of edge configurations in the Lightning Model, for which the following definitions are useful. 
\begin{definition}
	 A {\em path from a to b} in a configuration $\fep(\phi)$ is a sequence of distinct vertices $a=a_0, a_1, a_2, \ldots, a_{n-1}, a_n=b \in \mathbb{Z}^2$ such that $a_i$ is a neighbor of $a_{i+1}$ and $\fep(\phi)(a_i, a_{i+1}) = \, \to \,$ for all $0\le i<n$. In this case we write $a \tep b$. If both $a \tep b$ and $b \tep a$, we write $a \lrep b$, and say that $a$ and $b$ are {\em bi-directionally connected}. 
\end{definition}
Note that in the case $a\lrep b$, there is no requirement that the forward and backward paths are the reverse of each other.

\begin{definition}
	Let $a \in \mathbb{Z}^2$. We define the \emph{strongly-connected component} of $a$ in $\fep(\phi)$ to be
	\[ \{ b \in \mathbb{Z}^2 : a \lrep b\} \; .  \]
	
	This is also called the \emph{strongly-connected cluster of a}.
\end{definition}

For the remainder of this paper, whenever we discuss paths, clusters, etc., the vertex configuration $\phi$ is assumed to have been sampled from $\V$ and the relations $\tep$, $\lrep$ have been generated by $\fep(\phi)$ as above.
 
\begin{definition}
Let $a \in \mathbb{Z}^2$. If $\{b \in \Z^2: a \tep b\}$ is infinite, then we say that
$a$ \emph{percolates weakly}.
\end{definition}

It is clear from the definition that $\{\phi: a\text{ percolates weakly}\}$ is measurable. 
K\"onig's Lemma (\cite{Kon}) implies that $a$ percolates weakly if and only if there is an infinite path starting from $a$.

\begin{definition}
Let $a \in \mathbb{Z}^2$. If the strongly-connected component of $a$ is infinite, we say that $a$ \emph{percolates strongly}. 
\end{definition}

The set $\{\phi:a\text{ percolates strongly}\}$ is also measurable.
We use the phrase {\em weak} (resp., strong) {\em percolation} to mean weak (resp., strong) percolation at the origin. We denote the events of weak percolation and strong percolation by $\{ 0 \tep \infty \}$ and $\{ 0 \lrep \infty \}$, respectively.

\subsection{Basic results} \label{sec:mono}

Percolation in the Lightning Model is monotone in \ep. Here is the argument.  
We define the following partial order on $\E$: 

\begin{definition}  Given two edge configurations $v, w \in \E$, we say $v \preceq w$ if every open edge of $v$ is also open in $w$.  
\end{definition}

It is clear that the maps $\{f_\ep\}$ are monotone in \ep \, with respect to this partial order; that is, 

\begin{equation}\label{prop:mon}
	\tag{MP}
\ep \le \ep'  \; \implies  \fep(\phi) \preceq \fepp(\phi) \; \forall \; \phi\in \V. 
\end{equation} 

In other words if $\ep \, \le \ep'$ then the directed graph defined by $\fep(\phi)$ is a subgraph of $\fepp(\phi)$. The following immediate corollary is stated as a theorem because of its importance: 

\begin{theorem}\textbf{\em Monotonicity of Percolation Probability} \label{thm:monot} \medskip \\
Both $\Pep(\{\text{weak percolation}\})$ and 
$\Pep(\{\text{strong percolation}\})$ are non-decreasing functions of $\epsilon$.
\end{theorem}
\begin{proof}
Let $W_\ep = \{\phi\in\V\colon 0 \tep \infty\text{ in $\fep(\phi)$}\}$ and $S_\ep = \{\phi\in\V\colon
0 \lrep \infty\text{ in $\fep(\phi)$}\}$. It follows immediately from (\ref{prop:mon}) that for $\ep \le \ep'$, $W_\ep \subseteq W_\epp$ and $S_\ep \subseteq S_\epp$.   	
\end{proof}

We now show that 0 {\em  goes to infinity with the same probability that infinity comes to }0.
This result hints that weak percolation could force strong percolation. 

\begin{definition} Denote $\{v \in \Z^2: v \tep 0\}$ as the {\em attracting set}  of the origin. If this set is infinite, we write $\infty \tep 0$. 
\end{definition}
Observe that the event $\{\infty \tep 0\}$ is measurable. Recall that weak percolation is the event $\{0 \tep \infty\}$.

\begin{theorem}
\label{probability of infinite alpha and omega}
$\Pep(\{ 0 \tep \infty \}) = \Pep(\{ \infty \tep 0 \}).$
\end{theorem}

\begin{proof} 
Given an edge $e=(a,b)$, its {\em flip} is the edge $\bar e = (b,a)$. Define the {\em mirror transformation} $M: \E \to \E$ as sending each edge to its flip:  
\[ M(z)_e =z_{\bar e}, \quad z \in \E, \; e \in E. \]
 
Let $I$ be the involution on vertex configurations defined by $I:\V\to \V$, $I(\phi)_a=1-\phi_a$. It is easy to check that $f_\epsilon(I(\phi))_e=M(f_\epsilon(\phi))_e$. Recalling that $\lambda$ is Lebesgue product measure on $\V$, one also checks that  $\lambda\circ I^{-1}=\lambda$. It follows that
$$\Pep\circ M^{-1}=
\lambda\circ f_\epsilon^{-1}\circ M^{-1}=\lambda\circ (M\circ f_\epsilon)^{-1}=
\lambda\circ(f_\epsilon\circ I)^{-1} 
=\lambda\circ f_{\epsilon}^{-1}=\Pep,$$ so that $\Pep$ is invariant under $M$. To finish, observe that the mirror image of the event $\{0 \tep \infty\}$ is the event $\{\infty \tep 0\}$. 
%
%
%
\end{proof}

\section{Upper and Lower Bounds for Percolation} \label{sec:lowerweak}

In this section we show that for sufficiently small $\ep$, weak percolation almost surely fails to occur. We first give an elementary counting argument which shows that weak percolation fails when $\ep = 0$. Then we sketch a soft argument showing that weak percolation also fails for some positive $\ep$. This argument, however, does not give explicit non-trivial lower bounds for such $\ep$, so we include a third argument, based upon estimating the spectral radius for an appropriate linear transformation, which gives a non-trivial lower bound. We note that it is certainly not sharp. 

Here is an argument showing  that weak percolation fails when $\ep = 0$. Fix a path $0 \to a_1 \dots \to a_n$ in $\overrightarrow{G}$. This path is open in $f_0(\phi)$ if and only if $\phi_{0}>\phi_{a_1}>
\ldots >\phi_{a_n}$. This last event has probability $1/(n+1)!$ since the values $\phi_{0},\ldots,
\phi_{a_n}$ are chosen independently and uniformly from $[0, 1]$. But since paths are simple, there are at most $4\cdot 3^{n-1}$ paths of length $n$ starting from the origin. Hence the probability that at least one of the paths is open is bounded above by $4\cdot 3^{n-1}/(n+1)!$, which clearly tends to 0 as $n \to \infty$. Since weak percolation can occur only if there is an open path of length $n$ (from the origin) for each $n$, we see that weak percolation fails almost surely when $\ep = 0$. 

Here is a soft argument showing that weak percolation fails also for some positive $\epsilon$. Fix a non-self-intersecting path with $nk$ edges (where $k$ is to be fixed below and $n$ will be increased
to $\infty$) and vertices $a_0=0,\ldots,a_{nk}$. 
Then the probability that the path is open is bounded above by the probability that each of the paths $a_{ik}\to\ldots \to a_{ik+(k-1)}$ are
open for $i=0,\ldots,n-1$. Since these events are independent, the probability of this is $\mathbb P_\epsilon(\{a_0\to\ldots\to a_{k-1}\text{ is open}\})^n$. 
Next one sees  that $\mathbb P_\epsilon(\{a_0\to\ldots\to a_{k-1}\text{ is open}\})$ is a continuous function of $\epsilon$, converging to $1/k!$ as $\epsilon\to 0$. 
Now choose $k = 7$ (so that $k!>3^k$) and fix an $\epsilon>0$ sufficiently small so that
$\mathbb P_\epsilon(a_0\to\ldots\to a_6\text{ is open})<3^{-7}$. Then the standard counting argument shows that the probability there exists an open path of length $nk$ starting from the origin approaches 0 as $n\to\infty$.  

%

 Next we give a more detailed argument which obtains an explicit lower bound by computing the probability that a fixed path of length $n$ is open using an iterated linear operator whose spectral radius may be accurately estimated. 

\begin{theorem}
Let $\epsilon_0=0.1481$. In the Lightning Model with $\epsilon\le \epsilon_0$, the 
probability of weak percolation is zero.
\end{theorem}

\begin{proof}
Fix $\ep > 0$, and consider a fixed path in $\overrightarrow{G}$ consisting of vertices with potential values $x_0, x_1, x_2, \ldots, x_n$. The path will be open iff the values satisfy $x_0 \gep x_1 \gep x_2 \gep \cdots \gep x_n$, where we write $x_a \gep x_b$ to indicate $x_b < x_a + \epsilon$. Since the values at each of the vertices are chosen independently and uniformly, the probability of the set of vertex configurations satisfying $x_0 \gep x_1 \gep x_2 \gep \cdots \gep x_n$ is given by
\begin{equation}\label{eq:pval}	\int_0^1 dx_0
	\int_0^{\min(x_0+\epsilon, 1)}dx_1 
	\int_0^{\min(x_1+\epsilon, 1)}dx_2
	\ldots
	\int_0^{\min(x_{n-1}+\epsilon, 1)}dx_{n} \; .
\end{equation}

We  define now a related sequence of functions that lends itself to recursive evaluation: 
\begin{definition} Let $y \in [0,1]$. For $\ep > 0$ and $n = 0, 1, \dots$, define
\[ \Fep_n(y) = \int_0^{\min(y+\epsilon, 1)}dx_1 
						 \int_0^{\min(x_1+\epsilon, 1)}dx_2
						 \ldots
						 \int_0^{\min(x_{n-1}+\epsilon, 1)}dx_{n}.
\]
\end{definition}
$\Fep_n(y)$ gives the probability of being able to continue $n$ steps from a vertex having value $y$, i.e., it is the probability that there exists an open path of length $n$ beginning at a given vertex, conditioned on that vertex having potential value $y$. 

This last formula contains one fewer integral than appeared in expression \eqref{eq:pval},
and the probability that a fixed path of length $n$ is open is $\Fep_{n+1}(1)$.

\begin{example}\label{ex:polys}  It is enlightening to calculate the first few of these directly. Here we are assuming that $\ep < 1/2$. 
\[ \Fep_0(y) = 1. \]

\[ \Fep_1(y) = 
\begin{cases}
1 								& y \in (1-\epsilon, 1]; \\
\epsilon + y 					& y \leq 1-\epsilon.
\end{cases} \]

\[ \Fep_2(y) = 
\begin{cases}
- \frac{\epsilon^2}{2} + \epsilon + \frac{1}{2}
& y \in (1-\epsilon, 1]; \\
- \frac{\epsilon^2}{2} + 2\epsilon - \frac{1}{2} + y
& y \in (1-2\epsilon, 1-\epsilon]; \\
\frac{3 \epsilon^2}{2} + 2 \epsilon y + \frac{y^2}{2}
& y \leq 1-2\epsilon.
\end{cases} \]
These piecewise defined polynomials in $y$ are continuous, as the values agree at the endpoints of each subinterval of definition.
\end{example}

The endpoints of the intervals of the piecewise definition are all of the form $1-j\epsilon$, or 0, motivating the following definition.

\begin{definition} Let $M = \ceil{\frac{1}{\ep}}$. For $j = 0, 1, \ldots, M-2$, we define
\[I_j = \left( 1-(j+1)\epsilon, 1-j\epsilon \right],\]
and for $j = M-1$ define
\[I_{M-1} = [0, 1-(M-1)\epsilon].\]
Then $\{I_0, I_1, \dots, I_{M-1}\}$ gives a partition of [0,1] into $M$ subintervals of length $\ep$ with perhaps the exception of $I_{M-1}$, which has length at most $\ep$.  
\end{definition}

In Example \ref{ex:polys}, each of the calculated $\Fep_n(y)$, restricted to the interval $I_j$, is a polynomial in $y$ of degree $j$. We will now show that this pattern holds for all $n$ and $\ep \in (0, 1]$. We note that $\Fep_n(y)$ satisfies $\Fep_n(y) = \int_0^{\min(y+\epsilon, 1)} \Fep_{n-1}(x)dx$, which motivates the following.

Define a linear transformation on (suitable) functions by $$\Lep f(x) = \int_0^{\min(x+\epsilon, 1)} f(t) \, dt,$$ and set  \[ \Fspace = \big\{f: [0,1] \to \mathbb R: \ f\restr{I_j}\text{is a polynomial of degree }\leq j \big\}, \]
a finite dimensional vector space.

\begin{lemma} \label{lem:inv}
	For any $\epsilon>0$, $\Fspace$ is invariant under $\Lep$. 
\end{lemma}

\begin{proof}
	
	Let $s_j=1-j\epsilon$ so that $I_j=(s_{j+1},s_j]$ for $j<M-1$ and $I_{M-1}=[0,s_{M-1}]$. If $\1{I_j}$ denotes the indicator function of $I_j$, it is easily checked that if we define \[h_{j,i}(x) = \left( x-s_{j+1}\right)^i \cdot \1{I_j}(x), \]
	then the set of functions 
$$
\{h_{j,i}(x): j = 0, 1, \ldots, M-1 \text{ and } i = 0, 1, \ldots, j\}
$$
forms a basis for $\Fspace$. Hence the lemma will follow if we show that $\Lep h_{j,i}\in \Fspace$ for each $0\le j<M$ and $0\le i\le j$.
	
	We deal first with the case $j<M-1$. If $x\le s_{j+2}$, then $x+\epsilon\le s_{j+1}$ 
	so that $\big[0,\min(x+\epsilon,1)\big]$ does not intersect $I_j$. It follows that $\Lep h_{j,i}$ is identically 0 on $\bigcup_{k=j+2}^{M-1}I_k$. 
	
	If $x\ge s_{j+1}$, then $x+\epsilon\ge s_j$, so that 
	\begin{equation*}
		\int_0^{\min(x+\epsilon,1)}h_{j,i}(t)\,dt=\int_{I_j}h_{j,i}(t)\,dt=\frac{\epsilon^{i+1}}{i+1},
	\end{equation*}
	which is independent of $x$. That is, the restriction of $\Lep h_{j,i}$ to $\bigcup_{k=0}^{j}I_k$ is
	a degree 0 polynomial. 
	
	If $x\in (s_{j+2},s_{j+1}]$, then it is straightforward to calculate
	\begin{equation*}
		\int_0^{\min(x+\epsilon,1)}h_{j,i}(t)\,dt = \int_{s_{j+1}}^{x+\epsilon}h_{j,i}(t)\,dt = \frac 1{i+1}(x-s_{j+2})^{i+1}.
	\end{equation*}
	That is, the restriction of $\Lep h_{j,i}$ to $I_{j+1}$ is $\frac{1}{i+1}h_{j+1,i+1}$. Combining these, we have shown that
	\begin{equation}\label{eq:smallj}
		\Lep h_{j,i}=\frac{1}{i+1}\left(\sum_{k=0}^j \epsilon^{i+1}
	h_{k,0}+h_{j+1,i+1}\right).
	\end{equation}
	
	Similarly if $j=M-1$ and $x\in [0,1]$, then $x+\epsilon>s_{M-1}$, and we have
	\begin{equation*}
		\Lep h_{M-1,i}(x) =\int_{I_{M-1}}h_{M,i}(t)\,dt = \int_{0}^{1-(M-1)\epsilon}\big(t-(1-M\epsilon)\big)^i\,dt, 
	\end{equation*}
	which shows that
	\begin{equation}\label{eq:bigj}
		\Lep h_{M-1,i}=\frac{1}{i+1}\big(\epsilon^{i+1}-(M\epsilon-1)^{i+1}\big) = \frac{\epsilon^{i+1}-(M\epsilon-1)^{i+1}}{i+1}\sum_{k=0}^{M-1} h_{k,0}.
	\end{equation}
\end{proof}

\begin{remark}
Even though each $h_{j,i}$ is not continuous, one can easily check that after applying $\Lep$ each resulting function is continuous. Also notice that since $M=\lceil \frac1\epsilon\rceil$, $0\le M\epsilon-1<\epsilon$, so that all the coefficients of the matrix representing $\Lep$ with respect to this basis are non-negative.

\end{remark}

\begin{corollary}
$\Fep_n(y) \in \Fspace$ for every $n \in \N$.
\end{corollary}

\begin{proof}
$\Fep_0(y) = 1\in \Fspace$ by definition. The result follows by induction, since $\Fep_n(y) = \left((\Lep)^n1\right)(y)$. \end{proof}

In order to estimate the growth rate of $\Fep_{n+1}(1)$, the probability a fixed path of length $n$ is open, we may work entirely within $\Fspace$. Fix the ordered basis $S = \{h_{0,0} ,\, h_{1,0} ,\, h_{1,1} , h_{2, 0}, h_{2, 1}, h_{2, 2}, \, \ldots ,\, h_{M-1,M-1}\}$ for $\Fspace$ and let $A$ be the matrix representing $\Lep$ with respect to $S$.  $\Fep_0$ is constant function $1$ on the interval $[0,1]$, which we can write as the linear combination $ 1 = 1 h_{0,0} + 1 h_{1,0} + \ldots + 1 h_{M-1,0}$. Hence, its coordinate vector with respect to our ordered basis is $ \begin{bmatrix} 1 & 1 & 0 & 1 & 0 & 0 & \cdots & 0 \end{bmatrix}^T$. Since $\Fep_n= (\Lep)^n \, \Fep_0 $,  the coefficients of the function $\Fep_n$ with respect to the basis are given by
\[A^n
\begin{bmatrix} 1 & 1 & 0 & 1 & 0 & 0 & \cdots & 0 \end{bmatrix}^T. \]
We want to evaluate the function this represents, at the value $y =1$. Evaluating at $y=1$ means we are only interested in the values of the resulting function on the interval $I_0$, which is given by the coefficient of $h_{0,0}$, which we can get by simply taking the first entry of the previous matrix product. 

That is,
\begin{equation} \label{eq:probeval} 
\Fep_n(1)=\begin{bmatrix} 1 & 0 & 0 & \cdots & 0 \end{bmatrix}
\left( A^n
\begin{bmatrix} 1 & 1 & 0 & 1 & 0 & 0 & \cdots & 0 \end{bmatrix}^T \right) \, .
\end{equation}
It follows that 
\begin{equation}
	\label{eq:matnorm} \Fep_{n+1}(1) \le C \|A^{n+1}\| \le C \|A^n\|, 
\end{equation}
for some constant $C$ (which may depend upon $\ep$).

Let us consider our current position. We wish to show that for some small  $\ep > 0$, weak percolation does not occur (a.s.). The value $\Fep_{n+1}(1)$ is the probability that a fixed path of length $n$ is open. If we let $\mu_n$ denote the number of paths of length $n$ (starting at 0, say) in $\mathbb Z^2$, it is sufficient to  show 
\begin{equation} \label{eq:sufficient} 
 \lim_{n \to \infty} \; \mu_n \cdot \Fep_{n+1}(1) \, = \, 0 \,. 
\end{equation}
By sub-multiplicativity, $\mu_n^{1/n}$ is convergent; the limit is the \emph{connective constant},
$\lambda$. It is known (\cite{PT}) that $\lambda\le 2.679192495$.
%
Recall that the {\em spectral radius} of $A$ is given by $\rho(A)=\lim_{n\to\infty}\|A^n\|^{1/n}$, 
Moreover $\rho(A)$ is just the maximum of the absolute values of the eigenvalues of $A$.
%
Now if $\lambda \cdot \rho(A) < 1$
then \eqref{eq:sufficient} follows. Thus, to establish \eqref{eq:sufficient}  it is sufficient to find $\ep_0 > 0$ for which $\rho(A) <  0.373246 <  1/\lambda$. 
Finally, by monotonicity,  weak percolation would fail a.s.\ for each $0 \le \epsilon \le \ep_0$.  

Using equations \eqref{eq:smallj} and \eqref{eq:bigj}, we may easily compute the matrix entries for $A$.  When $\epsilon_0=0.1481$, $A$ is a $28\times 28$ matrix whose spectral radius is approximately $\rho(A)\approx 0.373079$ so that $\lambda\cdot\rho(A)<1$ as required.
\end{proof}

\begin{proposition} \label{prop:sup-nice}
When $\epsilon$ is greater than $p_\text{site}$, the critical probability for the standard site percolation model
in $\Z^2$, the Lightning Model has positive probability of strong percolation.
\end{proposition}

\begin{proof}
For $\phi \in \V$ and $\ep > 0$ set $S = S_{\phi, \ep} = \{a \in \mathbb{Z}^2 : \phi_a>1-\epsilon\}$. By \ep-tolerance, if $a$ and $b$ are neighbors in $S$, then both edges $(a,b)$ and $(b,a)$ are present in $\fep(\phi)$.
In particular, if $C$ is a (non-directed) cluster in $S$,  then $C$ is contained in a single strongly connected component
of $\fep(\phi)$. 
	
Fix $\epsilon>p_\text{site}$.
Since the vertex potentials are independently uniformly distributed, each $a\in\Z^2$ belongs
to the random set $S$ with probability $\epsilon$ independently of all other vertices. 
Let $A$ be the set of vertex configurations such that $S$ is infinite. By standard site
percolation, $\lambda(A)>0$. By the previous paragraph, $\fep^{-1}\{0\lrep\infty\}
\supset A$, so that $\Pep(\{0\lrep\infty\})>0$ as required.
%

\end{proof}

Wierman (\cite{WI}) established that $p_{\text{site}} < 0.679492$, which therefore gives an upper bound for the critical threshold for strong percolation in the lightning model. 


\section{Number of Infinite Components} \label{sec:numcom}

In this section, we show that for the Lightning Model,  the number of infinite (strong) 
clusters is almost surely 0, 1, or $\infty$. By the results in the previous section, for 
sufficiently large $\ep$ the Lightning Model strongly percolates, so in that case there is 
at least one cluster. Due to ergodic considerations, we know that the number of clusters 
is almost surely constant. However, we cannot at this time rule out the possibility that 
there are infinitely many infinite strong clusters.

\begin{definition}
For $\omega \in \E$, let $\Nom$ denote the number of infinite strong
clusters in the configuration $\omega$.
\end{definition}

We remark that $\Nom$ is a measurable function of $\omega$. 
For a natural number $n$, let $B_n$ denote $\{  (x,y) \in \Z^2 : \max\{|x|, |y|\} \le n\}$. 
For $r\in\N$ and $\omega\in\E$, which we think of a subgraph of $\overrightarrow{G}$,
the \emph{restriction} of $\omega$ to $B_r$, written $\omega|_{B_r}$, denotes the
induced subgraph of $\omega$ on the vertex set $B_r$.
For $\omega\in\E$ and $m<n<r$, let $C(m,n,r)(\omega)$ denote the number of clusters in  
$\omega|_{B_r}$ that intersect both $B_m$ and $B_n^c$. 
\begin{lemma}
	For $\omega\in\E$, the limits
	\begin{align*}
		&\lim_{r\to\infty}C(m,n,r)(\omega),\\
		&\lim_{n\to\infty}\lim_{r\to\infty}C(m,n,r)(\omega)\text{ and}\\
		&\lim_{m\to\infty}\lim_{n\to\infty}\lim_{r\to\infty}C(m,n,r)(\omega)
	\end{align*}
	all exist (the first for all $n>m>0$ and the second for all $m>0$).
	Also $N_\omega=\lim_{m\to\infty}\lim_{n\to\infty}\lim_{r\to\infty} C(m,n,r)(\omega)$,
	so that $\omega\mapsto N_\omega$ is measurable.
\end{lemma}

\begin{proof}
	For now, let $n>m>0$ be fixed. 
	For $r>n$, let $\leftrightarrow_r$ be the equivalence relation on $B_n$ where $u\leftrightarrow_r v$
	if there is a directed path from $u$ to $v$ in $\omega|_{B_r}$ and a directed path from $v$ to $u$ in 
	$\omega|_{B_r}$. These equivalence relations are increasing. That is, if $u\leftrightarrow_r v$, then
	$u\leftrightarrow_{r'}v$ for all $r'>r$. Since there are finitely many equivalence relations on $B_n$,
	they stabilize at some $r_0>n$. From that point onwards, the sequence 
	$(C(m,n,r)(\omega))_{r\ge r_0}$ does
	not change.
	(Prior to this point, $C(m,n,r)(\omega)$ may increase as new
	 connections added in the outer layer can ensure that
	a cluster connects $B_m$ to $B_n^c$ (in a bidirectional way);
	or decrease as new connections added in the outer layer 
	can merge two previously existing clusters). It follows that
	$\lim_{r\to\infty}C(m,n,r)(\omega)$ exists. We see that
	$\lim_{r\to\infty}C(m,n,r)(\omega)$ is the number of clusters in $\omega$ that intersect both $B_m$ and
	$B_n^c$.
	
	Given this, it is clear that $\lim_{r\to\infty}C(m,n,r)(\omega)$ is a non-increasing function of $n$ (as if 
	a cluster intersects $B_{n'}^c$ for $n'>n$, then it intersects $B_n^c$), so that
	$\lim_{n\to\infty}\lim_{r\to\infty}C(m,n,r)(\omega)$ exists.
	This is the number of infinite clusters in $\omega$ that intersect $B_m$ as a cluster is infinite if and only if
	it intersects each $B_n^c$. 
	
	Finally we see $\lim_{n\to\infty}\lim_{r\to\infty}C(m,n,r)(\omega)$ is an increasing function of
	$m$, so the
	limit as $m$ approaches $\infty$ also exists, possibly taking the value $\infty$. 
	From the above, we see that 
	$\lim_{m\to\infty}\lim_{n\to\infty}\lim_{r\to\infty} C(m,n,r)(\omega)$ is the number
	of infinite clusters in $\omega$ as required.
	
	Since $C(m,n,r)(\omega)$ is a measurable function of $\omega$ (as it
depends on finitely many coordinates), it follows that $N_\omega$ is measurable as required.
\end{proof}

Here is the main theorem of this section: 

\begin{theorem} 
\label{no 2 infinite clusters}
For each $\epsilon>0$, there exists $k\in\{0,1,\infty\}$ such that for
$\Pep$-a.e.\ $\omega\in\E$, $\Nom=k$. 
\end{theorem}

For the proof of Theorem \ref{no 2 infinite clusters}, we introduce a transformation on vertex configurations that modifies the potential values only for vertices in a large finite box centered the origin. The idea is that as a result of applying the transformation, a large sub-box (again centered at the origin) will be forced to be strongly connected in every configuration. 

For $n \in \N$, recall that $B_{2n}=\{(x,y)\in\Z^2\colon \max(|x|,|y|)\le 2n\}$. Within such a box, define the sequence of \emph{layers} $L_0, L_1, \ldots$, $L_{2n}$ by
\[ L_i = \lbrace (x,y) \in \mathbb{Z}^2 : \max (|x|,|y|) =2n-i \rbrace . 
\]

These are illustrated in Figure \ref{fig:layers}.

We now define our transformation. For $n \in \mathbb{N}$ and $\eta \in (0,1)$, define $\Ph^\eta: \V \to \V$ by
\[ \Ph^\eta (\phi)(a)=
\begin{cases}
(1-\eta)^i \phi(a) & \ \ \ \text{if $a \in L_i \,$;}\\
\phi(a) & \ \ \ \text{otherwise,}
\end{cases} \]
for any vertex configuration $\phi \in \V$ and vertex $a \in \mathbb{Z}^2$. The following lemma describes the useful properties of $\Ph^\eta$. 
\begin{lemma}\label{lem:PsiProps}
Let $\epsilon>0$ be given. The transformations $\Ph^\eta$ defined above have the following properties:
\begin{enumerate}
\item For each $n \in \N$ and $\eta \in (0,1)$, if $A\subset\V$ has positive measure, then $\Ph^\eta(A)$ also has positive measure. \label{it:positive}
\item Suppose $n > \log( \tfrac 1 \ep)$ and set $\eta = \log( \tfrac 1 \ep)/n$. For each $\phi\in \V$, $\fep(\Ph^\eta(\phi))$ has the property that\label{it:bicenter}
all of the edges within the central sub-box, $B_n$, are bidirectionally connected. In particular, the strongly connected component of the origin contains $B_n$. 
\item Suppose $n > \tfrac 1 \ep$ and let $\eta$ be as in \ref{it:bicenter}. There exists a universal constant $\delta>0$, independent of $\ep$ and $n$, so that the probability that
$\Ph^\eta$ doesn't break any edges is at least $\delta$. That is, \label{it:nobreak}
\[
\lambda\Big(\big\{\phi\in\V\colon \fep(\phi)\preceq \fep(\Ph^\eta(\phi))\big\}\Big)\ge\delta.
\]
\end{enumerate}
\end{lemma}
\begin{figure}[h]
	\label{layers}
	\caption{Layers $L_0, L_1, \ldots$ in the case $n = 5$.}
	\centering
	\includegraphics[width=0.8\textwidth]{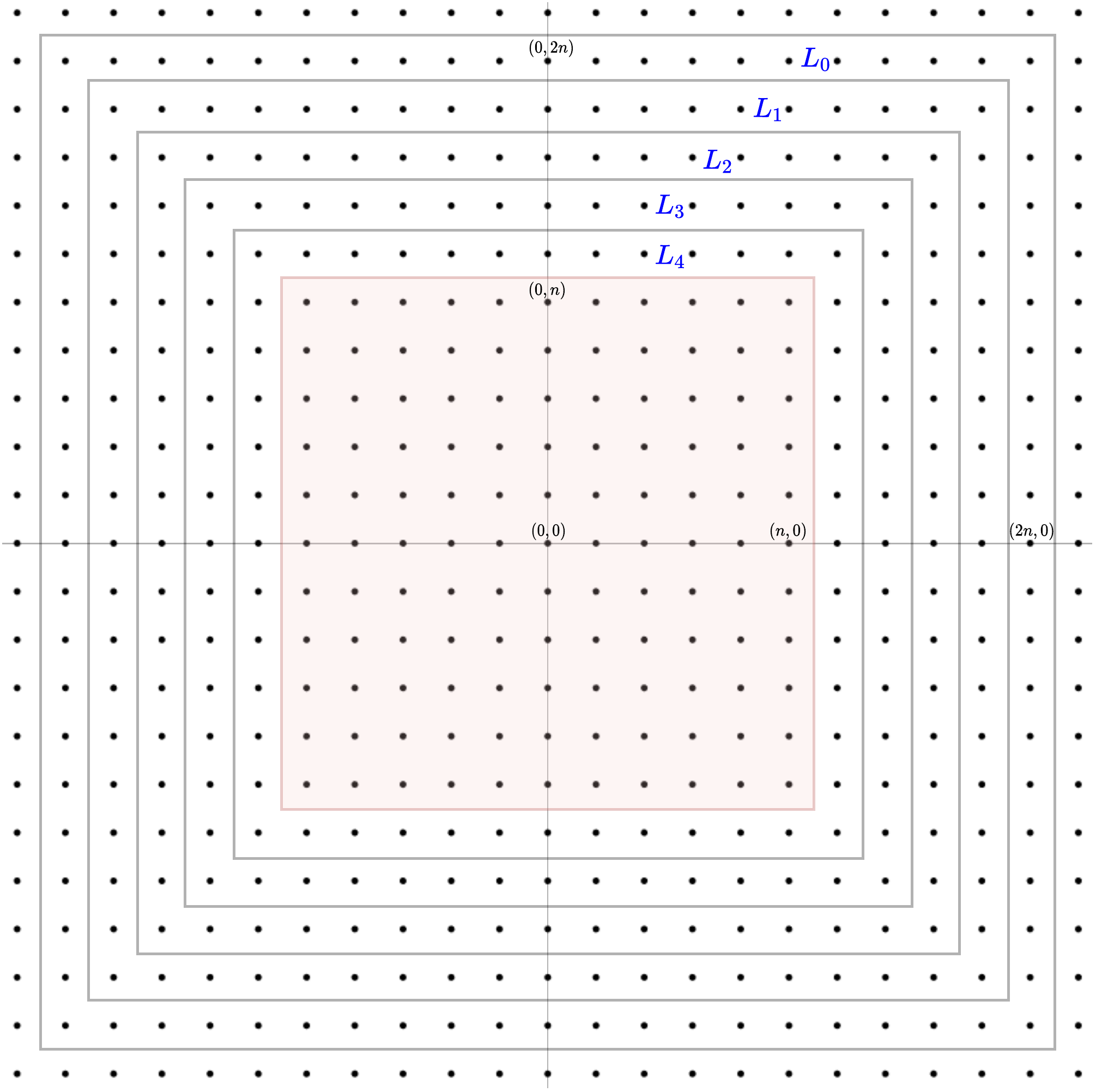}\label{fig:layers}
\end{figure} 
%
\begin{proof}[Proof of Lemma \ref{lem:PsiProps}:]
Condition \ref{it:positive} follows since $\Ph^\eta$ is simply scaling the values of a finite number of coordinates. 

To establish condition \ref{it:bicenter}, we show that $(1-\eta)^n<\ep$. This ensures that for all $\phi\in\V$,
the potential values of $\Ph^\eta(\phi)$ at vertices in $B_n$ are all less than $\epsilon$. By the argument
in Proposition \ref{prop:sup-nice}, this ensures that the edges contained in $B_n$ are bidirectionally
connected in $\fep(\Ph^\eta(\phi))$.

Notice that $(1-\eta)^n<\ep$ follows immediately from the fact that for $\eta \in (0, 1)$, $1-\eta <  e^{-\eta}$
and hence $(1-\eta)^n <  e^{-n\eta}=\epsilon$. 

We move to the proof of \ref{it:nobreak}. Define\footnote{This order induces the partial order $\preceq$ on edge configuations in the obvious way.}  an order $\sqsubseteq$ on the set $\{\, \raisebox{2pt}{\tikz\draw[black,fill=black] (0,0) circle (.3ex);}\,,\rightarrow\}$  by \; 
$ \raisebox{2pt}{\tikz\draw[black,fill=black] (0,0) circle (.3ex);}\, \sqsubseteq \, \raisebox{2pt}{\tikz\draw[black,fill=black] (0,0) circle (.3ex);}\,\text{, } 
\rightarrow\ \sqsubseteq\  \rightarrow \text{, and }
\, \raisebox{2pt}{\tikz\draw[black,fill=black] (0,0) circle (.3ex);}\, \sqsubseteq \ \rightarrow \;$. Let $a, b \in \mathbb{Z}^2$ be any two adjacent vertices, and let $\phi \in \V$ be some vertex configuration. We say that the edge $e=(a,b)$ is \emph{not broken} by $\Ph^\eta$ if $\fep(\phi)_e \sqsubseteq \fep(\Ph^\eta(\phi))_e$.

First notice that if $a,b \in \mathbb{Z}^2$ are adjacent vertices in the same layer $L_i$, then the edge $(a,b)$ is never broken by $\Ph^\eta$. This is because if both vertex values are scaled by the same positive value, then the difference between them will be scaled by that value as well, and the scaling is a contraction. It is worth noting  that this may create new open edges, something we will need to consider later in the proof of Theorem \ref{no 2 infinite clusters}. 

If $a$ and $b$ are adjacent vertices such that $a\in L_i$ and $b\in L_{i+1}$ (i.e. $b$ is 
closer to the origin), we refer to the edge $(a,b)$ as an \emph{inwards} edge (similarly
$(b,a)$ is an \emph{outwards} edge).  
We claim that if $(a,b)$ is an inwards edge and $a\to b$ in $\fep(\phi)$, then $a\to b$ in 
$\fep(\Ph^\eta(\phi))$. Since $a\to b$ in $\fep(\phi)$, we have $\phi(b)\le \phi(a)+\epsilon$. 
This implies $\Ph^\eta(\phi)(b)=(1-\eta)^{i+1}\phi(b)\le (1-\eta)^{i+1}\phi(a)+(1-\eta)^{i+1}\epsilon
\le (1-\eta)^i\phi(a)+\epsilon=\Ph^\eta(\phi)(a)+\epsilon$, so that $a\to b$ in $\fep(\Ph^\eta(\phi))$ also.

We now know that the only edges that might not be preserved by $\Ph^\eta$ are the outwards edges.
We will narrow this down even further, proving that only edges in an outer ``ring" near the boundary of $B_n$ might not be preserved. This will be useful because it represents a small quantity of the total number of edges in $B_n$. 

\begin{claim}
Let $e = (a,b)$ be an outwards edge with $a \in L_{i+1}$ and $b \in L_i$ and $a\to b$
in some configuration $\fep(\phi)$. If $i \ge \frac{1}{\epsilon}$ 
then $a\to b$ in $\fep(\Ph^\eta(\phi))$.

That is, for $i\ge \frac 1\epsilon$, no edges from level $i+1$ to level $i$ are broken. 
\end{claim}
\begin{proof}[Proof of Claim:] 
By assumption $\phi_b<\phi_a+\epsilon$ and $i\ge 1/\epsilon$.
We simply need to verify that $(1-\eta)^i\phi_b<(1-\eta)^{i+1}\phi_a+\epsilon$. 

We compute
\begin{align*}
(1-\eta)^i\phi_b&<(1-\eta)^i\phi_a+(1-\eta)^i\epsilon\\
&=(1-\eta)^{i+1}\phi_a+\epsilon+\eta(1-\eta)^i\phi_a-\epsilon(1-(1-\eta)^i)\\
&=(1-\eta)^{i+1}\phi_a+\epsilon+(1-\eta)^i\left[\eta \phi_a-\epsilon \left(\frac 1{(1-\eta)^i}-1\right)\right].
\end{align*}

Since $\frac 1{1-\eta}>1+\eta$, we see $\frac 1{(1-\eta)^i}>1+i\eta$ so that
$\eta \phi_a-\epsilon\big(\frac{1}{(1-\eta)^i}-1\big)<\eta\phi_a-\epsilon i\eta=\eta(\phi_a-\epsilon i)<0$.
Hence, the term displayed above in square brackets is negative, so that
\[
(1-\eta)^i\phi_b<(1-\eta)^{i+1}\phi_a+\epsilon,
\]
as required.
\end{proof}

There is now a very small collection of edges which might not be preserved, the outwards ($L_{i+1} \to L_i$) edges in the outermost $\frac{1}{\epsilon}$ layers ($i$ ranging between  0 and $\floor{\frac{1}{\ep}}$). We move to estimating the probability that they might be broken. 

We call an edge between a vertex in $L_1$ and a vertex in $L_0$ an \emph{outer} edge. We begin by obtaining an upper bound on  the probability that an outer edge will not  be preserved, and then use this to get an upper bound for the probability that any edge is not preserved.

\begin{claim}
Let $e=(a,b)$ be an outer edge. 
$\Pep(e \text{ is broken}) \leq \frac{\eta}{2}$.
\end{claim}

\begin{proof}[Proof of Claim:]
Let $a \in L_1$ and $b \in L_0$ be neighbors in $\Z^2$. The probability that the edge $(a,b)$ is not preserved by $\Ph^\eta$ is the probability that the following two inequalities hold:
\begin{align*}
\phi_a &> \phi_b - \epsilon;\\
\Ph^\eta(\phi)(a) &\not> \Ph^\eta(\phi)(b) - \epsilon.
\end{align*}
Let $s=\phi_a$ and $t=\phi_b$. 
The above conditions can then be written as
\begin{equation}\label{eq:constr}
\begin{split}
t &< s + \epsilon;\\
t &\geq (1-\eta) s + \epsilon.
\end{split}
\end{equation}
The probability we want is the area of the region in $[0,1]^2$ defined by these two inequalities. While the exact value is slightly messy to compute, we may obtain a useful overestimate if we allow $t$ to extend to its maximum value for $s \in [0,1]$. This region is a triangle with vertices at $(0,\epsilon)$, 
$(1,1+\epsilon)$ and $(1,1-\eta+\epsilon)$, so it has width 1 and height $\eta$, and therefore
area $\frac\eta 2$. Hence, for an outer edge $e$, $\Pep(e \text{ is broken})\le\frac\eta 2$, as claimed. \end{proof}


%
%

We now move to the case of a general outwards edge from $L_{i+1} \to L_i$, for $1 \le i \le \floor{ \frac{1}{\ep}}$. 
In order for an edge $(a,b)$ with $a\in L_{i+1}$ and $b\in L_i$ to not be preserved, the constraints
\eqref{eq:constr} become
\begin{align*}
t &< s +\epsilon;\\
t &\geq (1-\eta) s + (1-\eta)^{-i}\epsilon.
\end{align*}
Since the lower bound has been increased (as $\frac1{1-\eta}>1$),
 we see that the area of the set of solutions is smaller than it
was for the outer edges from $L_1$ to $L_0$. Hence, for a general outward edge between $L_{i+1}$ and $L_i$, the probability that it is broken is less than $\eta/2$.

Summarizing: we know that the only edges which might be broken are outwards edges in  the outer $\floor{\frac{1}{\ep}}$ layers, and that the probability that any one of them is broken is less than $\eta/2$.

We move next to calculating the probability that no edge is broken by $\Ph^\eta$. To do so, we consider separately the {\em corners} and {\em sides} of the outermost \fle layers of $B_n$.  Let $C_1,\ldots,C_4$ denote the four corners of $B_n$ of size $\lfloor \frac 1\epsilon\rfloor
\times \lfloor \frac 1\epsilon\rfloor$ and let $S_1,\ldots,S_4$ denote the sides of $B_n$, that is,
the regions of size $\lfloor \frac1\epsilon\rfloor\times (4n+1-2\lfloor\frac1\epsilon\rfloor)$ and
$(4n+1-2\lfloor\frac1\epsilon\rfloor)\times\lfloor \frac1\epsilon\rfloor$ between the corners. These
regions contain the only edges that could be broken by applying $\Ph^\eta$. 

%

Because there are fewer than $\frac 1{\epsilon^2}$ outward edges in each corner, and the the probability that any given 
outward edge is broken is at most $\frac\eta2$, the probability that some edge is broken is bounded above by $\frac{\eta}{2\ep^2}$ and hence the probability that {\em no} edge is broken in any given corner by $\Ph^\eta$ is at least $1-\frac\eta{2\epsilon^2}$.

Next, we move to the sides. We intend to show that within each side, the probability that no edge is broken by $\Ph^\eta$ is at least $(1-\frac\eta{2\epsilon})^{4n}$. For this estimate we cannot simply use the union bound as we did for the estimate in the corners; doing this naively gives an upper bound for the probability that an 
edge is broken which is greater than 1! The difficulty we face is the built-in dependence of the edges in the lightning model: while the vertex potential values are independent, the existence of an edge $a\to b$ affects the probability of an edge $b\to c$. We want to use independence in some fashion, however. To obtain our desired estimate, we split each side into disjoint pieces for which the event that some edge is broken in
one piece is independent of the event that some edge is broken in another piece. 

Decompose each side $S_i$ into $4n+1-2\lfloor \frac 1\epsilon\rfloor$ disjoint {\em slices}, that is, 
outward paths of length $\lfloor 1/\epsilon\rfloor$ which have one vertex in each 
of the layers $L_0, L_1, \ldots L_{\floor{1/\epsilon}}$. The union of these slices contain all of the outward 
edges in $S_i$. In each slice, the union bound implies the probability than some edge is broken is at most 
$\frac{\eta}{2\epsilon}$, hence the probability that no edge is broken in a given slice
is at least $1-\frac{\eta}{2\epsilon}$. For each slice, consider the event that no edge is broken. These events are independent. Thus, we see the probability that no edge is broken in $S_i$ is at least 
$(1-\frac\eta{2\epsilon})^{4n+1-2\lfloor\frac 1\epsilon\rfloor}$.

Since the $C_i$'s and the $S_j$'s are mutually disjoint, the above argument shows that the probability that 
no edges are broken in any of the corners or sides is at least
\[
\left(1-\frac\eta{2\epsilon^2}\right)^4
\left(1-\frac\eta{2\epsilon}\right)^{16n}.
\]
Recalling that $\eta=(\log\frac 1\epsilon)/n$, the probability that no edges are broken is at least
\[
\left(1-\frac{\log\frac1\epsilon}{2n\epsilon^2}\right)^4
\left(1-\frac{\log\frac 1\epsilon}{2n\epsilon}\right)^{16n}.
\]
Since $\epsilon$ is assumed to be fixed, then by using the well-known result that $(1+\frac an)^n\to e^a$ we see
that, for large $n$, the lower bound for the probability that no edges are broken approaches
\[
e^{-8\log\frac{1}{\epsilon}/\epsilon}=\epsilon^{8/\epsilon},
\]
so that the probability that there no edges are broken is bounded below uniformly in $n$.
\end{proof}

\begin{proof}[Proof of Theorem \ref{no 2 infinite clusters}]
Fix $\epsilon > 0$ so that strong percolation occurs. By ergodicity, $N_\omega$ is constant, $\mathbb P_\epsilon$-almost surely.
Suppose for a contradiction that there are almost surely exactly $k$ infinite strongly-connected components with $k\in\N\setminus\{1\}$. 
	
Suppose $n > \log \tfrac 1 \ep$ and let $\delta>0$ be as given in Lemma \ref{lem:PsiProps} so that for all such $n$,
\[\lambda\Big(\big\{\phi\in\V\colon \fep(\phi)\preceq \fep(\Ph^\eta(\phi))\big\}\Big)\ge\delta.
\]

By continuity of measures, we may choose $N>0$
so that for any $n \ge N,$
\[
\lambda\left(\{\phi \in \V\colon \text{$B_n$ intersects each infinite component in $\fep(\phi)$}\}\right) > 1-\delta. 
\]

Now fix  $n \ge \max \{\log \tfrac 1 \ep , N\}$ and consider the two events whose probabilities are given in the previous two inequalities. Since the sum of the probabilities is larger than 1, the intersection $E$  of these two events has positive probability. Since the image under  $\Ph^\eta$ of a set of positive measure still has positive measure (Condition \ref{it:positive} of Lemma \ref{lem:PsiProps}), we have $\lambda(\Ph^\eta(E))>0$.   

Suppose $\phi \in E$. Then $\fep(\phi)$ has $k$ infinite clusters, and the box $B_n$ must contain vertices from each of them.   Additionally, any edge of $\fep(\phi)$ will still be present in $\fep(\Ph^\eta(\phi))$ (possibly becoming bidirectional).  
It follows that in $\fep(\Ph^\eta(\phi))$, the $k$  infinite strongly-connected clusters originally present in $\fep(\phi)$ are contained in a single infinite cluster, which we will denote by $\mathcal C_*(\phi)$. 

It appears, then, that what we have created is a set of vertex configurations of positive measure ($\Ph^\eta(E)$) for which the resulting edge configurations have a single infinite cluster. However, in directed percolation it is possible to modify finitely many edges and create an infinite strong cluster where there was none before; see example \ref{ex:inffromfin} below. 

Hence, it is conceivable that $k-1$ additional infinite clusters
were created during the modification process, which would avoid the contradiciton that we seek.

We can deal with this issue by defining a new map $\widehat{\Ph^\eta}$ by
\[
\widehat{\Ph^\eta}(\phi)_a=\begin{cases}
\left(\Ph^\eta(\phi)\right)_a&\text{if $a\in \mathcal C_*(\phi))$;}\\
\phi_a&\text{otherwise.}
\end{cases}
\]


Since there are only finitely many possibilities for $\mathcal C_*(\phi)\cap B_n$, and all of them depend measurably on $\phi$, there is a fixed set $\Lambda\subset B_n$ such that $\mathcal C_*(\phi)\cap
B_n=\Lambda$ for all $\phi$ in a positive measure subset $A$ of $E$.  A {\em finite energy} argument similar to that given in the proof of Lemma \ref{lem:PsiProps}\eqref{it:positive} then shows that $\widehat{\Ph^\eta}(A)$ is of positive measure.  

Our final claim is that for any $\phi\in A$, $\widehat{\Ph^\eta}(\phi)$ has a unique infinite strongly connected cluster. To see this, let $\mathcal C$ be any cluster for the edge configuration $\fep(\widehat{\Ph^\eta}(\phi))$. 
If $\mathcal C$ intersects $\mathcal C_*(\phi)$, then $\mathcal C\supset\mathcal C_*(\phi)$
(note they may not be equal since they are generated from different potentials: $\widehat{\Ph^\eta}(\phi)$
and $\Ph^\eta(\phi)$). On the other hand, if $\mathcal C$ does not intersect $\mathcal C_*(\phi)$, then
the restriction of $\fep(\widehat{\Ph^\eta}(\phi))$ to $\mathcal C$ is the same as the restriction of $\fep(\phi)$
restricted to $\mathcal C$, so that $\mathcal C$ is a finite cluster. We have shown that for each $\phi \in A$, any cluster in $\fep(\widehat{\Ph^\eta}(\phi))$ is either contained in $\mathcal C_*(\phi)$ or is finite. Since $A$ has positive measure, this contradicts our original assumption that there were exactly $k$ infinite strong components almost surely, and we are done. 
\end{proof}

We remark that the proof in this section is essentially two-dimensional. 
Lemma \ref{lem:PsiProps}, part \ref{it:bicenter} makes essential use of the fact that 
$\eta=\Theta(1/n)$: this guarantees that after applying
$\Psi$, the central block is fully connected. On the other hand, Lemma \ref{lem:PsiProps}, part
\ref{it:nobreak} requires
that $\eta=O(1/n)$: this part ensures that no edges are broken when the potential is transformed
by $\Psi_n^\eta$. The proof works by showing that the only edges potentially broken are those within 
$1/\epsilon$ of the edge. In two dimensions, there are $O(n)$ edges that are potentially broken
and each has a probability of the order of $\eta$ of breaking, allowing us to show that the probability
that no edges are broken is $\Theta(1)$.

\section{Open Problems}\label{sec:open}

It is natural to ask whether one can rule out the case of infinitely many infinite strong clusters. 
The Burton-Keane theorem \cite{BK} is a well known approach to this in the case of undirected graphs. 
The following examples show that some things behave quite differently when dealing with directed graphs.

\begin{example} \label{ex:inffromfin}{\bf Creating an infinite strong cluster by modifying a single edge:}
Consider an edge configuration with edges only on two parallel lines, corresponding to, say, the lines $y = 1$ and $y = 2$ in $\Z^2$. Suppose the top line is a {\em source}, that is, each vertex $(x, 2)$ for $x \ge 0$ has a right-pointing edge to $(x+1, 2)$, and each vertex $(x, 2)$ for $x \le 0$ has a left-pointing edge to $(x-1, 2)$. Reverse the corresponding edges on the line $y = 1$, and at every fourth $x$-value, say, place a downward edge from $(x, 2)$ to  $(x, 1)$. This has no strongly connected infinite component, but adding a single edge from $(0, 1)$ to $(0, 2)$ creates one.
\end{example}

\begin{example}\label{ex:trifpts}{\bf Many splitting points give rise to the same boundary partition:}
Burton and Keane's proof works by considering \emph{``splitting points"}, that is places where when a
single vertex is removed from a configuration, an infinite cluster splits into at least three separate infinite clusters. 
The proof counts splitting points, showing that if they exist, their number grows 
proportionally to the volume of a region
by the ergodic theorem, while showing they grow at most proportionally to the surface area of a region.
For the upper bound they study, in a large volume, which boundary points belong to which infinite component
when a splitting point is removed. Critically, removing different splitting points gives rise to
a different infinite component structures. Unfortunately in the directed case, the removal of many 
splitting points may give rise to the same infinite component structures.

This is illustrated schematically in Figure \ref{fig:BKfail}: when any of the (red) splitting points is removed, 
the only infinite clusters are the bi-directional paths connecting the circle to infinity.
\end{example}

\begin{figure}[h!]
\begin{center}\includegraphics[scale=.23]{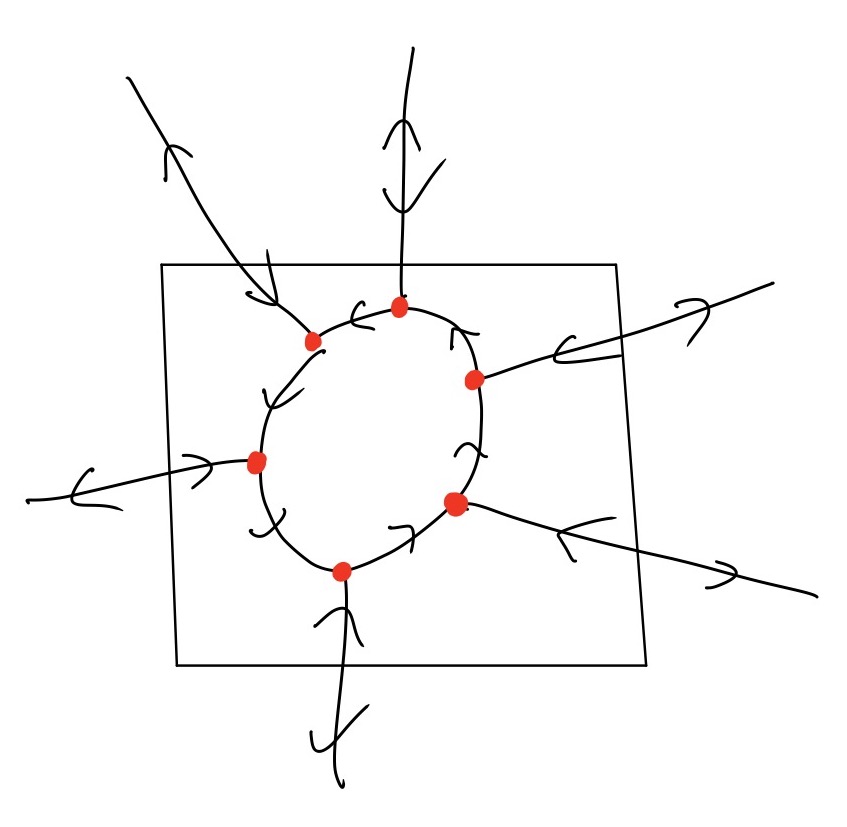}\end{center}
\caption{A failure of Burton-Keane in directed graphs}\label{fig:BKfail}
\end{figure}

Two more questions of interest: \begin{enumerate}
	\item Does positive probability of weak percolation imply positive probability of strong percolation? We conjecture that it does. 
	\item The results prior to Section 5 carry over to higher dimensions.  However, the proof of Theorem \ref{no 2 infinite clusters} depends upon working in two dimensions. Does this result hold for $d \ge 2$? 
\end{enumerate}

{\bf Data Availability Statement:} Data sharing not applicable to this article as no datasets were generated or analysed during the current study.

\bibliographystyle{amsalpha}
\bibliography{percolation}

\Addresses
\end{document}